\documentclass[a4paper, 12pt, leqno]{amsart}

\usepackage{amsmath,amsthm,amsfonts,amssymb,graphicx,mathrsfs,esint,longtable,mathtools, todonotes,enumitem,verbatim}
\usepackage[colorlinks=true, linkcolor=red, urlcolor=red, citecolor=green]{hyperref}
\usepackage{cleveref}
\usepackage{color}

\newcommand{\R}{{\mathbb R}}
\newcommand{\N}{\mathbb{N}}

\newcommand{\Ss}{\mathbb{S}}
\newcommand{\Ha}{\mathcal{H}}
\newcommand{\e}{\varepsilon}
\newcommand{\al}{\alpha}
\newcommand{\la}{\lambda}
\newcommand{\modu}{\operatorname{mod}}
\newcommand{\diam}{\operatorname{diam}}

\newtheorem{theorem}{\textbf{THEOREM}}[section]
\newtheorem{lemma}[theorem]{\textsc{Lemma}}
\newtheorem{proposition}[theorem]{\textsc{Proposition}}
\newtheorem{corollary}[theorem]{\textsc{Corollary}}
\theoremstyle{remark}

\theoremstyle{definition} 
\newtheorem{definition}[theorem]{\textsc{Definition}}
{\theoremstyle{remark} }

\usepackage{xpatch}
\makeatletter
\AtBeginDocument{\xpatchcmd{\@thm}{\thm@headpunct{.}}{\thm@headpunct{}}{}{}}
\makeatother

\def\XXint#1#2#3{{\setbox0=\hbox{$#1{#2#3}{\int}$ }
\vcenter{\hbox{$#2#3$ }}\kern-.6\wd0}}

\usepackage[shortcuts]{extdash}
\allowdisplaybreaks

\def\charfn_#1{{\raise1.2pt\hbox{$\chi_{\kern-1pt\lower3pt\hbox{{$\scriptstyle#1$}}}$}}}
\def\leq{\leqslant }
\def\geq{\geqslant }

\def\XXint#1#2#3{{\setbox0=\hbox{$#1{#2#3}{\int}$}
\vcenter{\hbox{$#2#3$}}\kern-.5\wd0}}

\def\le {\leqslant}
\def\ge {\geqslant}

\begin{document}
\title[Quasisymmetric Koebe uniformization]{Quasisymmetric Koebe uniformization with weak metric doubling measures}
\author{Kai Rajala and Martti Rasimus}
\let\thefootnote\relax
\thanks{The authors were supported by the Academy of Finland, project number 308659. }

\begin{abstract} 
We give a characterization of metric spaces quasisymmetrically equivalent to a finitely connected circle domain. This result generalizes the uniformization of Ahlfors $2$-regular spaces by Merenkov and Wildrick \cite{MW13}.
\end{abstract}
\maketitle

\section{Introduction} \label{intro}
A homeomorphism $f$ between metric spaces $(X,d)$ and $(Y,d')$ is \emph{quasisymmetric} if there exists a homeomorphism $\eta \colon [0,\infty) \to [0,\infty)$ such that
$$ \frac{d'(f(x),f(y))}{d'(f(x),f(z))} \le \eta \left( \frac{d(x,y)}{d(x,z)} \right)$$
for all distinct points $x,y,z \in X$. Quasisymmetric maps form a natural generalization of conformal maps to the setting of abstract metric spaces. In particular, the \emph{uniformization problem} for quasisymmetric maps is important due to applications in areas such as geometric group theory, complex dynamics, and geometric topology. 

The uniformization problem asks which spaces admit quasisymmetric maps onto some \emph{standard} space such as $\mathbb{S}^2$. Bonk and Kleiner \cite{BK02} solved the problem for \emph{Ahlfors $2$-regular spheres} $(X,d)$, i.e., topological spheres for which the two-dimensional Hausdorff measure $\Ha_d^2$ satisfies  
$$ 
C^{-1} r^2 \le \Ha_d^2(B_d(x,r)) \le Cr^2 \quad \text{for all } x \in X, \, 0<r<\diam X. 
$$
Bonk and Kleiner showed that \emph{linear local connectedness} (see Section \ref{bwm}) is a necessary and sufficient condition for $2$-regular spheres to be quasisymmetrically equivalent to $\mathbb{S}^2$. Here and in what follows, we equip $\mathbb{S}^2$ with the standard spherical metric. 

Later, Merenkov and Wildrick \cite{MW13} considered the \emph{multiply connected} setting, generalizing the classical Koebe uniformization. They gave a characterization 
for the finitely connected Ahlfors $2$-regular surfaces that are quasisymmetrically equivalent to circle domains in $\mathbb{S}^2$, and generalized it to countably connected surfaces with some additional geometric properties. Here a \emph{circle domain} is an open and connected set whose set of complementary components 
consists of disks and points. We refer to \cite{MW13} and the related work by Bonk \cite{Bon11} for further motivation and background. 

Our aim is to find similar characterizations for surfaces that need not be $2$-regular, such as \emph{fractal surfaces}. Uniformization results for fractal surfaces are of great importance 
in view of applications, cf. \cite{BM17}, \cite[Section 2]{MW13}, but one cannot expect results as strong as above to hold. 

In \cite{LRR18}, Lohvansuu and the authors introduced the \emph{weak metric doubling measures}, generalizing the metric doubling measures, or Strong $A_\infty$-weights, of David and Semmes \cite{DS90}. These are, roughly speaking, measures that can be used to construct quasisymmetric deformations for a given metric, see Section \ref{bwm} for the precise definition. We then gave a version of the Bonk-Kleiner theorem in terms of the existence of such measures. 

In this paper we apply the weak metric doubling measures to finitely connected surfaces. Namely, we have the following generalization of the characterization given 
by Merenkov and Wildrick.

\begin{theorem} \label{mainthm}
Let $X$ be a metric space homeomorphic to a domain in $\Ss^2$ such that $\overline X \setminus X$ contains finitely many components. Then $X$ is quasisymmetrically equivalent to a circle domain if and only if it is linearly locally connected, carries a weak metric doubling measure and $\overline X$ is compact.
\end{theorem} 
Here $\overline{X}$ is the completion of $X$. The ``only if" part of Theorem \ref{mainthm} follows from the definitions in a straightforward manner. Theorem \ref{mainthm2} below is a quantitative version of the ``if" part. 

To prove this, we first apply the weak metric doubling measure to suitably deform the metric on $X$. We show that the deformed space is \emph{reciprocal} in the sense of \cite{R17}, and therefore admits a \emph{quasiconformal} map into $\mathbb{S}^2$ by a recent result of Ikonen \cite{I19}. We then apply geometric estimates to show that this map, when suitably normalized, is quasisymmetric. Some of these are related to the results in \cite{Bon11}, see Section \ref{jub} below. Our approach is different from those in \cite{MW13} and \cite{LRR18}, both of which apply the Bonk-Kleiner theorem.  


\section{Weak metric doubling measures}  \label{bwm}
For $x,y \in X$ and $\delta>0$, a finite sequence of points $x_0,x_1,\dots, x_m$ in $X$ 
is a \emph{$\delta$-chain from $x$ to $y$}, if $x_0=x$, $x_m=y$ and $d(x_j,x_{j-1}) \le \delta$ for every $j=1,\dots, m$. Notice that in every connected metric space each pair of points can be connected by a $\delta$-chain for any $\delta >0$.

Recall that a measure $\mu$ in a metric space $(X,d)$ is \emph{doubling} if there is $C_D\geq 1$ such that 
$$
\mu(B_d(x,2r)) \leq C_D\mu(B_d(x,r)) \quad \text{for all } x \in X, \, r>0. 
$$
From now on we assume that $\mu$ is a Radon measure in $X$ that is doubling with constant $C_D$. 

In what follows, we use notation 
$$
B_{xy}=B_d(x,d(x,y)) \cup B_d(y,d(x,y)). 
$$
Given $\mu$ and a ``dimension" $s>0$, we define the \emph{$\mu$-length} $q_{\mu,s}$ of points $x,y \in X$ as follows: set 
$$
q_{\mu,s}^\delta(x,y) \coloneqq \inf \Big\{ \sum_{j=1}^m \mu(B_{x_j x_{j-1}})^{1/s} \colon (x_j)_{j=0}^m \text{ is a $\delta$-chain from $x$ to $y$} \Big\}
$$
and 
$$
q_{\mu,s}(x,y) \coloneqq \limsup_{\delta \to 0} q_{\mu,s}^\delta(x,y) . 
$$

\begin{definition} \label{WMDMdef}
We say that $\mu$ is a \emph{$C_W$-weak metric doubling measure}, or \emph{WMDM}, of dimension $s>0$ in $(X,d)$, if for all $x, y \in X$, 
\begin{equation}
\label{WMDMestimate}
\frac{1}{C_W}  \mu(B_{xy})^{1/s} \le q_{\mu,s}(x,y). 
\end{equation}
\end{definition}

From now on we assume that $\mu$ is a $C_W$-WMDM of dimension $2$, and we abbreviate $q=q_{\mu,2}$. See \cite{LRR18} for examples and further discussion.  

Weak metric doubling measures should be compared to the \emph{metric doubling measures} of David and Semmes. They are essentially defined by 
requiring that in addition to \eqref{WMDMestimate} also the reverse inequality holds.

Given $\lambda \geq 1$, a metric space $(X,d)$ is $\lambda$-\emph{linearly locally connected}, or \emph{LLC}, if for any $x\in X$ and $r>0$, 
\begin{itemize}
\item[(i)] if $y,z \in B_d(x,r)$ then there exists a continuum $K \subset B_d(x, \la r)$ with $y,z \in K$, and 
\item[(ii)] if $y,z \in X \setminus B_d(x,r)$ then there exists a continuum $K \subset X \setminus B_d(x,r/\la)$ with $y,z \in K$.
\end{itemize}

From now on we assume that $(X,d)$ is $\lambda$-LLC and homeomorphic to a circle domain such that $\overline{X}$ is compact and $\overline{X}\setminus X$ contains $M<\infty$ components. We denote by $C_X$ the ratio of the diameter of $(X,d)$ to the minimum distance between the components of $\overline{X}\setminus X$.  We are now ready to state the main result of this paper. 

\begin{theorem} 
\label{mainthm2}
There is an $\eta$-quasisymmetric homeomorphism from $(X,d)$ onto a circle domain $\Omega \subset \mathbb{S}^2$, where $\eta$ depends only on $\lambda$, $C_X$, $C_D$, $C_W$, and $M$. 
\end{theorem}

As pointed out in the introduction, Theorem \ref{mainthm} is a straightforward consequence of Theorem \ref{mainthm2}. We do not know if the dependence on the number of components $M$ and the constant $C_X$ is necessary in Theorem \ref{mainthm2}, and if it admits extensions to countably connected domains corresponding to \cite[Theorem 1.4]{MW13}. The rest of the paper is dedicated to the proof of Theorem \ref{mainthm2}.


\section{Deformation of the metric}
Theorem \ref{mainthm2} is proved by showing that the $\mu$-length $q$ is a metric on $X$ with strong geometric properties. Our approach is based on the following reverse inequality for WMDMs. 

\begin{proposition}\label{muuka}
For every $x \in X$ there is $r_x>0$ such that 
$$
q(x,y) \le C_S \mu(B_{xy})^{1/2} 
$$ 
for all $y \in B_d(x,r_x)$, where $C_S=16C_WC_D^{28+16\lceil \log_2 \la \rceil}$. 
\end{proposition}

Before proving Proposition \ref{muuka}, we state some consequences. We will apply the following elementary property of doubling measures, see \cite[13.1]{H01}: For all $x \in X$ and  $0<r\le R < \diam_d(X)$, 
\begin{equation}
\label{doubite}
\frac1C \left( \frac{R}r \right)^{1/\al} \le \frac{\mu(B_d(x,R))}{\mu(B_d(x,r))} \le C \left( \frac{R}r \right)^\al. 
\end{equation}
Here $C$ and $\alpha$ depend only on $C_D$. 

\begin{corollary}
$(X,q)$ is a metric space homeomorphic to $(X,d)$. 
\end{corollary}

\begin{proof}
Combining Proposition \ref{muuka} and the first inequality in \eqref{doubite} shows that the identity map $\operatorname{Id}:(X,d) \to (X,q)$ is locally H\"older continuous. 
Similarly, combining Definition \ref{WMDMdef} with the second inequality in \eqref{doubite} proves the continuity of $\operatorname{Id}^{-1}$.   
\end{proof}

We use notations $B_d$ and $B_q$ for the open balls in $(X,d)$ and $(X,q)$, respectively. We next give estimates for measures of balls in $(X,q)$. 

\begin{lemma}
\label{kissa}
Let $x \in X$ and $s>0$. Then 
\begin{equation}
\label{clash}
\mu(B_q(x,s)) \leq C_W^2 s^2. 
\end{equation}
Moreover, if $r_x>0$ is as in Proposition \ref{muuka} and $B_q(x,s)\subset B_d(x,r_x)$, then 
\begin{equation}
\label{clash2}
\frac{s^2}{2C_S^2C_D} \leq \mu(B_q(x,s)). 
\end{equation}
\end{lemma}
\begin{proof}
First, we apply the WMDM-definition \ref{WMDMdef} to establish the inclusions 
\begin{eqnarray*} 
B_q(x,s) &\subset& \{y: \, \mu(B_{xy})^{1/2} <C_Ws\} \\
&\subset& \{y: \, \mu(B_d(x,d(x,y)))^{1/2} < C_Ws\}=B_d(x,r_s) 
\end{eqnarray*}
for some $r_s>0$. Since $\mu(B_d(x,r_s))\leq C_W^2 s^2$, \eqref{clash} follows. 
Similarly, Proposition \ref{muuka} and doubling yield 
\begin{eqnarray*}
B_q(x,s) &\supset& \{y: \, C_S\mu(B_{xy})^{1/2}<s \} \\ 
&\supset& \{y: \, C_D^{1/2}C_S \mu(B_d(x,d(x,y)))^{1/2} < s\}, 
\end{eqnarray*}
from which \eqref{clash2} follows. 
\end{proof}

It follows from the above estimates that $\mu$ is in fact comparable to the $2$-dimensional Hausdorff measure $\Ha_q^2$ in $(X,q)$. We normalize $\Ha_q^2$ 
so that it coincides with the Lebesgue measure if $q$ is the euclidean metric in $\R^2$. 

\begin{corollary}
We have 
\begin{equation}
\label{nort}
\frac{1}{2\pi C_S^2C_D^4} \Ha_q^2(E) \leq \mu(E) \leq \frac{C_W^2}{\pi} \Ha_q^2(E) 
\end{equation}
for all Borel sets $E \subset X$. In particular, 
\begin{equation}
\label{nort2}
 \Ha_q^2(B_q(x,s)) \leq 2\pi C_S^2C_D^4C_W^2 s^2 
\end{equation}
for all $x \in X$ and $s>0$.  
\end{corollary}

\begin{proof}
The second inequality in \eqref{nort} follows directly from \eqref{clash} and the definition of $\Ha^2_q$. Also, \eqref{nort2} follows directly from \eqref{clash} and the first inequality in \eqref{nort}. 

For the first inequality in \eqref{nort}, we may assume that $E$ is open since $\mu$ is Radon. Given $\delta>0$, we can apply the $5r$-covering lemma to cover $E$ with 
balls $B^j_q(x_j,s_j) \subset E$ satisfying \eqref{clash2} such that the balls $B^j_q(x_j,s_j/5)$ are pairwise disjoint and each $s_j< \delta$. We denote the corresponding $\delta$-content by $\Ha^2_{q,\delta}$. 
Then by \eqref{clash2}, the doubling property of $\mu$, and the disjointness give 
\begin{eqnarray*}
\Ha^2_{q,\delta}(E) &\leq& \pi \sum_j s_j^2 \leq 2\pi C_S^2C_D \sum_j \mu(B_q(x_j,s_j))\\  
&\leq& 2\pi C_S^2C_D^4 \sum_j  \mu(B_q(x_j,s_j/5)) \leq 2\pi C_S^2C_D^4 \mu(E). 
\end{eqnarray*}
The claim follows by taking $\delta \to 0$. 
\end{proof}


\section{Proof of Proposition \ref{muuka}}

We prove Proposition \ref{muuka} by constructing a continuum connecting the given points with controlled $q$-diameter. We define the $q$-diameter with
$$\diam_q(A) = \sup_{a,b \in A} q(a,b)$$
for $A \subset X$, which makes sense even though we have not yet proved that $q$ is a finite distance. Note also that the definition of $q$ implies that it satisfies the standard triangle inequality.

As a first step of the construction we find separating continua in small annuli. We denote $\ell=\lceil \log_2 \la \rceil$ for the rest of this section.

\begin{lemma} \label{annuluslemma}

Let $x \in X$ and $r>0$ such that $\overline B_d(x,(2\la)^7r)$ is compact and contained in a topological disk $U \subset X$. Then there exists a continuum $K \subset \overline B_d(x,(2\la)^6r) \setminus B_d(x,2 \la r)$ separating $B_d(x,r)$ and $X \setminus \overline B_d(x,(2\la)^7r)$ with
\begin{equation}
\label{viiki}
\diam_q(K) \le 8C_WC_D^{12+4\ell} \mu(B_d(x,r))^{1/2}. 
\end{equation}
\end{lemma}

\begin{proof}
We use notation $S_d(x,r)=\{y \in X: d(x,y)=r\}$. Let 
\begin{equation}
\label{rollar} 
E=S_d(x,(2\la)^3 r) \quad \text{and} \quad F=S_d(x, (2\la)^4 r). 
\end{equation} 
By \eqref{WMDMestimate} and a standard compactness argument (see \cite[4.1]{LRR18}) there exists $\delta_{x,r}>0$ such that for all $y \in E$, $z \in F$ and $0<\delta < \delta_{x,r}$
\begin{equation}
\label{alaraja}
q^\delta(y,z) \ge \frac1{2C_WC_D} \mu(B_{yz})^{1/2}.
\end{equation}

Fix $0<\delta<\min ( \delta_{x,r}, r)$. Using the doubling property of $\mu$ and the $5r$-covering lemma, we can find a cover 
\begin{equation}
\label{rusina}
\mathcal B = \{B_1^i \}_{i=1}^m = \{ B_d(x_i,r_i) \}_{i=1}^m
\end{equation}
for the annulus 
$$
A=B_d(x,(2\la)^5 r) \setminus \overline B_d(x,(2\la)^2 r)
$$ 
such that the balls $B_d(x_i,r_i/5)$ are pairwise disjoint, $r_i<\delta/2$ and 
\begin{equation} \label{ennedy} 
\e^2 \le \mu(B_d(x_i,r_i/5)) \le C_D \e^2
\end{equation} 
for every $i$ and some fixed $\e>0$. 
Indeed, by \eqref{doubite} we can choose $\e$ small enough and $r_{\tilde{x}} < \delta/2$ for every $\tilde{x} \in A$ so that \eqref{ennedy} holds for $B_d(\tilde{x},r_{\tilde{x}}/5)$. 
Applying the $5r$-covering lemma to the family of such balls yields the desired $\mathcal{B}$. The balls in the cover are contained in $B_d(x,(2\la)^6 r) \setminus \overline B_d(x,(2\la) r)$ by the choice of $r_i$ and $\delta$.

If $z \in F$, there exists by the LLC-condition a continuum contained in $A$ that connects $z$ to $E$. Thus there is a chain of balls $B_1, \dots, B_n \in \mathcal B$ such that for some $y \in E$ we have $y \in B_1$, $z \in B_n$ and $B_j \cap B_{j+1} \ne \emptyset$ for every $j=1, \dots, n-1$. With this in mind, we define
$$\mathcal B_1 = \{ B \in \mathcal B : B \cap E \ne \emptyset \}$$
and
$$ \mathcal B_j = \Big\{ B \in \mathcal B \setminus \Big( \bigcup_{k=1}^{j-1} \mathcal B_k \Big) : B \cap \Big( \bigcup_{B' \in \mathcal B_{j-1}} B' \Big) \ne \emptyset \Big\}.$$
The collections $\mathcal B_j$ form layers selected from the cover $\mathcal B$, the first containing those balls that intersect $E$ and the subsequent ones those not previously selected which intersect with the previous layer.

Recall that each $z \in F$ is contained in some $B_z \in \mathcal B_n$, where $n$ depends on $z$. We claim that 
\begin{equation}
\label{risuna}
n \ge \frac{\sqrt m}{4C_WC_D^{6+\ell}} 
\end{equation}
for all such $z$, where $m$ is the number of balls in the cover \eqref{rusina}. Indeed, if $B_1, \dots , B_n$ is a chain of balls as above, then their centers and the points $y$ and $z$ form a $\delta$-chain. Thus, if we denote $B_0=B(y,r(B_1))$ and $B_{n+1}=B(z,r(B_n))$, then \eqref{alaraja}, the definition of $q^\delta(y,z)$, the doubling property of $\mu$, and \eqref{ennedy} yield 
\begin{eqnarray*}
 \mu(B_{yz})^{1/2} &\leq& 2C_WC_D q^{\delta}(y,z)  \leq 2C_WC_D\sum_{j=0}^{n+1} \mu(B_j)^{1/2} \\
&\leq& 4C_WC_D^{2}\sum_{j=1}^{n} \mu(B_j)^{1/2}  \leq 4C_WC_D^4 n \e.
\end{eqnarray*} 

But since $\lambda \geq 1$, we have $d(y,z) \geq (2\lambda)^4r-(2\lambda)^3r \geq 2^3\lambda^4r$ by our choices of these points and $E,F$ in \eqref{rollar}. 
Therefore, the triangle inequality yields 
\begin{equation} 
\label{kansu}
B_d(x,(2\la)^6 r) \subset B_d(y,2(2\la)^6r) \subset B_d(y,4(2\la)^2d(y,z)). 
\end{equation} 
As the $m$ balls $B_d(x_i,r_i/5)$ in $\mathcal B$ are pairwise disjoint, combining the doubling property of $\mu$ with \eqref{rusina}, \eqref{ennedy} and \eqref{kansu} yields 
$$ \mu(B_{yz}) \ge \frac{\mu(B_d(y,4(2\la)^2d(y,z)))}{C_D^{4+2\ell}}  \geq \frac{\mu(B_d(x,(2\la)^6r))}{C_D^{4+2\ell}} \ge \frac{m\e^2}{C_D^{4+2\ell}},$$
(recall that $\ell=\lceil \log_2 \la \rceil$) so \eqref{risuna} follows. 

Let $n_0 = \lceil \sqrt m / 4C_WC_D^{6+\ell} \rceil.$ As the layers $\mathcal B_j$ are pairwise disjoint, we have
\begin{align*}
\sum_{j=1}^{n_0} \sum_{B_i \in \mathcal B_j} \mu(B_i)^{1/2} & \le C_D^2 m \e 
 \le 4C_WC_D^{8+\ell}n_0\sqrt m \e \\
& \le 4C_WC_D^{11+4\ell} n_0 \mu(B_d(x,r))^{1/2}.
\end{align*}
Hence for some $1 \le j_0 \le n_0$ we have 
$$
\sum_{B_i \in \mathcal B_{j_0}} \mu(B_i)^{1/2} \le C' \mu(B_d(x,r))^{1/2}, 
$$ 
where $C'=4C_WC_D^{11+4\ell}$. We denote by $K_1'$ the compact set $\cup_{B_i \in \mathcal B_{j_0}} \overline B_i$. By the choice of $n_0$ and the LLC-condition $K_1'$ separates $B_d(x,r)$ and $X \setminus \overline B_d(x,(2\la)^7r)$. Moreover, since $\overline B_d(x,(2\la)^7r) \subset U$ for some $U\subset X$ homeomorphic to a disk, a component $K_1$ of $K_1'$ also separates the same sets (see for example \cite[V 14.3]{N61}).

By repeating the above construction for $\delta/j$, $j=2,3,\dots$ we obtain continua $K_j$, each separating $B_d(x,r)$ and $X \setminus \overline B_d(x,(2\la)^7r)$. By connectedness, between any two points of $K_j$ there exists a $\delta/j$-chain among the centers of the balls $B^i_j$ covering $K_j$. For each $j$ we have the same estimate
$$\sum_i \mu(B^i_j)^{1/2} \le C' \mu(B_d(x,r))^{1/2}.$$
Then using compactness in the Hausdorff metric for compact sets we find a subsequence of $(K_j)$ converging to a compact set $K'$. Now also $K'$ and hence one of its components $K$ again separates $B_d(x,r)$ and $X \setminus \overline B_d(x,(2\la)^7r)$.

If $a,b \in K$ and $\delta'>0$, we pick a large $j$ such that $\delta/j<\delta'$ and the Hausdorff distance between $K$ and $K_j$ is less than $\delta'$. Then from $K_j$ we find points $p_1, \dots, p_{l-1}$ so that $a$ and $b$ are connected by the $\delta'$-chain $a=p_0,p_1,\dots,p_{l-1},p_l=b$ with
$$q^{\delta'}(a,b) \le \sum_{i=1}^{l} \mu(B_{p_i p_{i-1}})^{1/2} \le 2C_DC' \mu(B_d(x,r))^{1/2}.$$
Since the upper bound holds for all $\delta'>0$ the estimate is true also for $q(a,b)$.
\end{proof}

For $x \in X$, $r>0$ and $K$ as in Lemma \ref{annuluslemma}, we let 
\begin{eqnarray*}
K(x,r) &=&K \quad \text{and} \\ 
\hat K(x,r) &=& \text{ the component of } X \setminus K(x,r) \text{ containing }x. 
\end{eqnarray*}
The following lemma on basic planar topology allows us to connect $K(x_1,r_1)$ and $K(x_2,r_2)$ for correctly chosen adjacent balls $B_d(x_1,r_1)$ and $B_d(x_2,r_2)$. We refer to \cite[5.1]{LRR18} for a proof. 

\begin{lemma} \label{leikkaa}
Let $x_1, x_2 \in X$ and $r_1, r_2 >0$ be as in Lemma \ref{annuluslemma} such that
\begin{enumerate}
\item $\hat K(x_1,r_1)$ and $\hat K(x_2, r_2)$ intersect,
\item $\hat K(x_1,r_1) \not \subset \hat K(x_2,r_2)$,
\item $\hat K(x_2,r_2) \not \subset \hat K(x_1,r_1)$.
\end{enumerate}
Then the continua $K(x_1,r_1)$ and $K(x_2,r_2)$ intersect.
\end{lemma}

With these lemmas we are ready to construct the desired continuum between the given points.

\begin{proof}[Proof of Proposition \ref{muuka}]

Let $y \in X$ be such that $B_d(x,2\la d(x,y))$ is contained in a topological disk. Then, as in the proof of Lemma \ref{annuluslemma}, we can cover the ball $B_1=B_d(x,\la d(x,y))$ with $M_1$ balls 
$$
B_1^i=B_d(z^1_i,r^1_i)
$$ 
such that $z^1_i \in B_1$, the balls $\frac15B_1^i$ are pairwise disjoint, and
\begin{equation} \label{rajat}
\mu(B_1)/4C_D^{8+7\ell} \le \mu((2\la)^7 B_1^i) \le \mu(B_1)/4C_D^{7+7\ell}
\end{equation}
for each $i$. The doubling condition and \eqref{rajat} now imply that 
\begin{equation}
\label{haahu}
M_1 \le C_D^{19+14\ell}  
\end{equation} 
and that 
$$
(2\la)^7r^1_i \leq \frac{\la \operatorname{radius}(B_1)}4. 
$$ 
Furthermore, Lemma \ref{annuluslemma} can be applied with $z^1_i$ and $r^1_i$ for each $i$.

Let $\mathcal I_1$ be the set of indices $i$ such that $B_1^i$ intersects the component $D_1$ of $B_1$ containing $x$ and $\hat K(z^1_i,r^1_i) \not \subset \hat K(z^1_j,r^1_j)$ for all $j \ne i$. For future reference, notice that $y \in D_1$ by the LLC-condition. Then the compact set 
$$
K_1 = \bigcup_{i \in \mathcal I_1} K(z^1_i,r^1_i) \subset 2B_1 
$$ 
is connected. Indeed, if $k,l \in \mathcal I_1$, there exists a path from $\hat K(z^1_k,r^1_k)$ to $\hat K(z^1_l, r^1_l)$ in $B_1$. This path is now covered by a chain of sets $\hat K(z^1_i,r^1_i),$ $i \in \mathcal I_1$, so that for consecutive members in the chain the corresponding continua $K(z^1_i,r^1_i)$ intersect by Lemma \ref{leikkaa} and the choice of $\mathcal I_1$.

Next we choose $h \in \mathcal I_1$ such that $x \in \hat K(z^1_h,r^1_h)$ and denote 
$$
B_2=B_d(z^1_h,(2\la)^7 r^1_h). 
$$ 
Then $x \in B_2$ and $2B_2 \subset \frac12B_1$. We cover $B_2$ with $M_2$ balls 
$$
B_2^i=B_d(z^2_i,r^2_i)
$$ 
so that all the properties above remain valid with the balls $B_1^i$ replaced by the balls $B_2^i$. In particular, \eqref{rajat} takes the form 
$$
\mu(B_2)/4C_D^{8+7\ell} \le \mu((2\la)^7 B_2^i) \le \mu(B_2)/4C_D^{7+7\ell}. 
$$
Repeating the previous construction then yields continuum 
\begin{equation}
\label{kippa}
K_2 = \bigcup_{i \in \mathcal I_2} K(z^2_i,r^2_i) \subset 2B_2. 
\end{equation}
We next show that 
\begin{equation}
\label{kippa2}
K_1 \cap K_2 \neq \emptyset.  
\end{equation}

First, if $K(z^2_i,r^2_i)$ is one of the continua in \eqref{kippa}, then $\hat K(z^1_h,r^1_h) \not \subset \hat K(z^2_i,r^2_i)$, since otherwise we would have 
$$
B_d(z^1_h,r^1_h) \subset B_d(z^2_i,(2\la)^7 r_i^2)
$$ 
and by \eqref{rajat}
$$\mu(B_d(z^1_h,r^1_h)) \le \mu(B_2)/4C_D^{7+7\ell} \le \mu(B_d(z^1_h,r^1_h))/4, 
$$
a contradiction. 

Secondly, if 
$$
w \in \overline{\hat K(z^1_h,r^1_h)} \cap K(z^1_h,r^1_h)
$$ 
then at least one of the sets $K(z^2_i,r^2_i)$ in \eqref{kippa} satisfies $w \in \hat K(z^2_i,r^2_i)$. Then also  
$$
\hat K(z^2_i,r^2_i) \not \subset \hat K(z^1_h,r^1_h) \quad \text{and} \quad \hat K(z^1_h,r^1_h) \cap \hat K(z^2_i,r^2_i) \neq \emptyset.
$$ 
Thus 
$$
K(z^1_h,r^1_h) \cap K(z^2_i,r^2_i) \neq \emptyset 
$$ 
by Lemma \ref{leikkaa}, and \eqref{kippa2} follows. 

We continue the above process to obtain continua 
$$
K_j = \bigcup_{i \in \mathcal I_j} K(z^j_i,r^j_i) \subset 2B_j 
$$
for each $j \in \N$, such that $K_j \subset 2B_j \ni x$ for all $j$, and 
$$
\diam_d(B_j) \to 0 \quad \text{as } j \to \infty. 
$$ 
Moreover, applying the constructions of the balls $B_j$ together with estimates \eqref{rajat} applied to these balls, we get 
\begin{equation}
\label{kaik}
\mu(B_{j+1})^{1/2} \le \frac12 \mu(B_j)^{1/2}. 
\end{equation}
Repeating the argument in the previous paragraphs, we see that 
$K_j \cap K_{j+1} \neq \emptyset$ for all $j$. Therefore,  
$$
K=\cup_{j=1}^\infty K_j \cup \{x \}
$$
is a continuum. 

We now apply \eqref{viiki} and the construction of the set $K$ to estimate its $q$-diameter. First, if $a,b \in K_1$, then for some $i_1,\dots,i_{m+1} \in \mathcal I_1$ and points $x_1,\dots,x_m \in K_1$ we have $a\in K(z^1_{i_1},r^1_{i_1})$,
\begin{eqnarray*}
x_1 \in K(z^1_{i_1},r^1_{i_1}) \cap K(z^1_{i_2},r^1_{i_2}), \ldots, x_m \in K(z^1_{i_m},r^1_{i_m}) \cap K(z^1_{i_{m+1}},r^1_{i_{m+1}}), 
\end{eqnarray*}
and $b \in K(z^1_{i_{m+1}},r^1_{i_{m+1}})$. By \eqref{viiki}, 
\begin{align*}  q(a,b) & \le q(a,x_1)+q(b,x_m) + \sum_{j=1}^{m-1} q(x_j,x_{j+1}) \\
& \le 8C_WC_D^{12+4\ell} \sum_{j=1}^{m+1} \mu(B(z^1_{i_j},r^1_{i_j}))^{1/2}. 
\end{align*}
Since $m+1 \leq M_1$, combining with \eqref{rajat} and \eqref{haahu} gives 
$$
q(a,b) \le C_2\mu(B_1)^{1/2}, 
$$
where $C_2=4C_WC_D^{28+15\ell}$. In particular, we get an upper bound for the $q$-diameter of $K_1$. Repeating the argument, we get 
$$
\diam_q(K_j) \le C_2\mu(B_j)^{1/2} 
$$ 
for all $j$. Combining with \eqref{kaik}, we moreover have 
\begin{equation}
\label{kraa}
\diam_q(K_j) \le 2^{1-j} C_2 \mu(B_1)^{1/2}
\end{equation}
for each $j$.

Now let $w_0 \in K_1$. Fix $\delta>0$, $\e>0$, and 
$$
w_j \in K_j \cap K_{j+1}
$$
for each $j \ge 1$. Since $d(w_j,x) \to 0$, we find $k \in \N$ such that $d(w_k,x)<\delta$ and $\mu(B_{w_k x})^{1/2}<\e$. Then, by \eqref{kraa}, 
$$
q^\delta(w_0,x) \le \sum_{j=1}^k \diam_q(K_j) + q^\delta(w_k,x) \le 2C_2 \mu(B_1)^{1/2} + \e
$$
and hence 
\begin{equation}
\label{inaal}
q(w_0,x) \le 2C_2 \mu(B_1)^{1/2}. 
\end{equation} 

Finally, recall that our goal is to bound $q(x,y)$. Since $y \in D_1$, we can repeat the argument above with the same cover for $B_1$ to find that \eqref{inaal} holds with $x$ replaced by $y$. 
By triangle inequality, we conclude that
$$q(x,y) \le 4C_2\mu(B_1)^{1/2} \le 4 C_D^\ell C_2 \mu(B_{xy})^{1/2}.$$
The proof is complete. 
\end{proof}


\section{Quasiconformal uniformization} 
\label{kuis}

Our strategy for proving Theorem \ref{mainthm2} is to apply the existence of a quasiconformal homeomorphism $f$ from $(X,q)$ to a circle domain $\Omega$. This is guaranteed by a recent result of Ikonen \cite{I19} and the classical Koebe uniformization of finitely connected Riemann surfaces. We will show in Sections \ref{jub} and \ref{ilee} that $f$ is in fact quasisymmetric, with respect to the original metric $d$, under a suitable normalization. 

We recall the geometric definition of quasiconformal maps. Let $Y=(Y,d)$ be a metric space such that $\Ha^2_d$ is finite on compact subsets. We moreover assume that $Y$ is a topological $2$-manifold. It then follows that $\Ha^2_d$ is positive on open sets, cf. \cite{R17}. 

Let $\Gamma$ be a family of paths in $Y$. We say that a Borel function $\rho \geq 0$ in $Y$ is \emph{admissible} for $\Gamma$, if 
$$
\int_{\gamma} \rho \, ds \geq 1 \quad \text{for all locally rectifiable } \gamma \in \Gamma. 
$$
The \emph{(conformal) modulus} of $\Gamma$ is 
$$
\modu(\Gamma)=\inf \int_{Y} \rho^2 \, d\Ha^2_d, 
$$
where the infimum is taken over all admissible functions. 

A homeomorphism $f:Y \to Z$ between spaces as above is (geometric) \emph{$K$-quasiconformal}, $K \geq 1$, if 
$$
K^{-1} \modu(\Gamma) \leq \modu(f\Gamma) \leq K \modu(\Gamma)
$$ 
for all path families $\Gamma$ in $Y$, where $f\Gamma=\{f \circ \gamma: \, \gamma \in \Gamma\}$. 

It is shown in \cite{R17} and \cite{R19} that if $Y$ is a topological disk for which there exists $C>0$ such that 
$$
\Ha^2_d(B_d(y,r))\leq Cr^2 \quad \text{ for all } y \in Y, \, r>0, 
$$ 
then there exists a $\pi/2$-qua\-si\-con\-for\-mal homeomorphism from $Y$ into the euclidean plane. Recently Ikonen \cite{I19} generalized this result to the case of non-simply connected surfaces. In particular, he showed that the upper bound \eqref{nort2} guarantees that there is a $\pi/2$-quasiconformal homeomorphism from our space $(X,q)$ onto a Riemann surface $Z$. Moreover, by the classical uniformization theorem for finitely connected Riemann surfaces, there is a conformal map from $Z$ onto a circle domain $\Omega$. Recall that conformal maps are $1$-quasiconformal 
in the sense of the geometric definition above, and that the composition of a $K_1$- and a $K_2$-quasiconformal map is $K_1K_2$-quasiconformal. Thus we have the following.  

\begin{proposition}
\label{quasiconf}
There is a $\pi/2$-quasiconformal homeomorphism $f: (X,q)\to \Omega$, where $\Omega \subset \mathbb{S}^2$ is a circle domain. If moreover $(X,q)$ is not homeomorphic to $\mathbb{S}^2$, 
then the statement remains valid with circle domain $\Omega \subset \mathbb{R}^2$. 
\end{proposition}

The second statement follows from the first simply by postcomposing $f$ with a suitable M\"obius transformation followed with the stereographic projection. 

\section{Modulus estimate in circle domains} \label{jub}
In this section we assume that $\overline{X}\setminus X$ has at least two components. Let $\Omega \subset \mathbb{R}^2$ be the circle domain in Proposition \ref{quasiconf}. 
We prove a Loewner-type modulus estimate, Proposition \ref{keyestimate}, which along with Proposition \ref{muuka} is the main technical result towards Theorem \ref{mainthm2}. 

Proposition \ref{keyestimate} is related to the work of Bonk \cite{Bon11} on the uniformization of Sierpi\'nski carpets in the plane. There a Loewner estimate 
is proved for domains whose boundary components are quasicircles that are suitably separated (\cite[Proposition 7.5]{Bon11}). The estimate in Proposition \ref{keyestimate}, 
which is also in terms of the separation of the boundary components, is more precise but only holds in the more restrictive setting of circle domains.  

In what follows, we denote by $\Gamma(A,B;G)$ the family of paths joining sets $A, B \subset \overline{G}$ in $G$, i.e., 
all the paths $\gamma:[a,b] \to \overline{G}$ such that $\gamma(a) \in A$, $\gamma(b) \in B$, and $\gamma(t) \in G$ for all $a<t<b$. We abbreviate 
$\modu(A,B;G)=\modu(\Gamma(A,B;G))$. 

\begin{proposition}
\label{keyestimate}
Let $E_1,E_2 \subset \Omega$ be disjoint continua such that  
\begin{equation}
\label{fars}
\frac{\min\{\diam(E_1),\diam(E_2) \}}{\operatorname{dist}(E_1,E_2)} \geq 1. 
\end{equation}
Then 
\begin{equation}
\label{cafe}
\modu(E_1,E_2;\Omega) \geq \frac{\alpha^M}{2\pi (10M)^M (M+2)^2},  
\end{equation}
where 
$$
\alpha = 2^{-2-2M-\pi^2C_W^2C_D^{1+\log_2 C_X}/8 \log 2}. 
$$
\end{proposition}

The rest of this section is devoted to the proof of Proposition \ref{keyestimate}. We denote the complementary components of $\Omega$ by 
$$
D_1, \ldots, D_M, \quad D_i=\overline{D}(z_i,r_i). 
$$ 
Complementary point-components do not have effect on the modulus. Therefore, 
we can assume that $r_i>0$ for all $i$. We use notation 
$$
\Delta(i,j)=\frac{\operatorname{dist}(D_i,D_j)}{\min\{r_i,r_j \}} 
$$
for the relative distances. The homeomorphism $f$ in Proposition \ref{quasiconf} uniquely extends to a bijection from the set of components of 
$\overline{X} \setminus X$ to the set $\{D_i\}$. We denote by $A_i$ the component corresponding to $D_i$ under this bijection.

\begin{lemma}
\label{separatedisks}
We have 
$\Delta(i,j) \geq \alpha$ for every $i \neq j$, where $\alpha$ is the constant in Proposition \ref{keyestimate}. 
\end{lemma}

\begin{proof}
Fix $i \neq j$ such that $r_i \leq r_j$. We consider $\modu(D_i,D_j;\Omega)$. We first claim that  
\begin{equation}
\label{mollo}
\modu(D_i,D_j;\Omega)\leq \frac{\pi}{2} \modu(A_i,A_j;X) \leq \frac{\pi C_W^{2}C_D^{1+\log_2 C_X}}{2}.   
\end{equation}
The first inequality follows from the quasiconformality of $f$. Towards the second inequality, recall that $C_X$ is the ratio of the diameter of $(X,d)$ to the minimum $d$-distance $D$ between the components $A_i$. Let $m \geq 1$ be the smallest integer such that $C_X \leq 2^m$. Then, by the WMDM-condition and the doubling property of $\mu$, the length of every path in $\Gamma(A_i,A_j;X)$ is at least 
$$
C_W^{-1}\inf_{x \in X} \mu(B(x,D))^{1/2} \geq C_W^{-1}C_D^{-m/2}\mu(X)^{1/2}. 
$$
Therefore, 
$$
\modu(A_i,A_j;X) \leq \int_X C_W^{2}C_D^{m}\mu(X)^{-1}\, d\mu= C_W^{2}C_D^{m}, 
$$
and \eqref{mollo} follows. We prove the lower bound for $\Delta(i,j)$ by showing that the opposite of \eqref{mollo} holds if $\Delta(i,j) < \alpha$.

Let $s=\operatorname{dist}(D_i,D_j)$ and 
$$
w = z_i+\frac{(r_i+\frac{s}2)(z_j-z_i)}{|z_j-z_i|}
$$ 
be the point in the middle of $D_i$ and $D_j$. If $\Delta(i,j) < \alpha$, we have
\begin{equation}
\label{foro}
2^{N+2}s \leq r_i, 
\end{equation}
where $N = \lfloor2M+\pi^2C_W^2C_D^{1+\log_2 C_X}/8 \log2\rfloor$. We consider the path families  
$$
\Psi_n = \{\text{components of } S(w,ts) \cap \overline \Omega : 2^{n-1} <t<2^n \}
 $$
for $n=1,\dots, N$. Every path in $\Psi_0 \cup \cdots \cup \Psi_{N}$ either connects $D_i$ and $D_j$ or intersects some $D_k$, $k \ne i,j$. We claim that any such $D_k=\overline D(z_k,r_k)$ intersects paths from at most two families $\Psi_n$.

Suppose to the contrary that $D_k$ intersects paths from $\Psi_n$ and $\Psi_{n+2}$ for some $n$. Then there exist $w_1,w_2 \in D_k$ with
$$ |w_1-w| < 2^{n}s \quad \text{and} \quad |w_2-w| > 2^{n+1} s$$
so that
\begin{equation}
\label{goro}
2^{n-1}s \le r_k.
\end{equation}
Using basic planar geometry, \eqref{foro} and \eqref{goro}, we have
$$
|z_k-z_i|^2 \le (r_i+\frac{s}2)^2+(r_k+2^{n-1}s)^2 <(r_i+r_k)^2.
$$
But this is impossible since $D_i$ and $D_k$ are disjoint.

Since the number of disks $D_k$, $k \ne i,j$ is at most $M$, we have shown that for at least $N-2M+2$ different indices $n$ all the paths in $\Psi_n$ connect $D_i$ and $D_j$. Using standard properties of the modulus we have then the lower bounds
$$  \modu(\Psi_n) \ge 4\modu(\{ S(0,t) : 1 < t < 2 \}) \ge \frac{2 \log 2 }{\pi}$$
for every $n$, see \cite[Theorem 10.12]{V71}, and
\begin{equation}
\label{mollo2}
 \modu(D_i,D_j;\Omega) \ge (N-2M+2)\frac{2 \log 2}{\pi} > \frac{\pi C_W^{2}C_D^{1+\log_2 C_X}}{4}. 
\end{equation}
We have thus proved that $\Delta(i,j) < \alpha $ leads to a contradiction with \eqref{mollo}, and the lemma follows.
\end{proof}

Recall that $D_i=\overline{D}(z_i,r_i)$. We next consider the sets 
$$
\Phi_i=\{1<t<1+\alpha: \, S(z_i,tr_i) \subset \Omega\},  
$$
where $\alpha$ is the constant in Proposition \ref{keyestimate}, and the family $\Gamma_i$ of all the (parameterized) circles $S(z_i,tr_i)$, $t \in \Phi_i$. 

\begin{lemma}
\label{hanska}
We have 
\begin{equation}
\label{harakka}
m_1(\Phi_i) \geq \frac{\alpha^M}{(10M)^M}
\end{equation}
for all $i=1, \ldots, M$. In particular, 
\begin{equation}
\label{harakka2}
\modu(\Gamma_i) \geq \beta = \frac{\alpha^M}{2\pi (10M)^M}. 
\end{equation} 
\end{lemma} 
\begin{proof}
Enumerate the disks according to decreasing radius, and fix $D_i$. By Lemma \ref{separatedisks}, $\operatorname{dist}(D_i,D_j)\geq \alpha r_i$ for every $j < i$. Now, if 
$$
\sum_{j \geq i+1} r_j \leq \alpha r_{i}/10, 
$$ 
then \eqref{harakka} holds. Otherwise $r_{i+1}\geq \alpha r_{i}/(10M)$. Continuing inductively, either 
\begin{equation}
\label{hungary}
\sum_{j \geq i+k+1} r_j \leq \alpha r_{i+k}/10, 
\end{equation}
for some $k$, or  
$$
r_{i+k+1}\geq \alpha r_{i+k}/(10M)\geq \ldots \geq r_{i}\alpha^{k+1}/(10M)^{k+1} 
$$
for all $k$. In the latter case,  
\begin{equation}
\label{mud}
r_j \geq r_i\alpha^{M-1}/(10M)^{M-1} \quad \text{for all } j=1, \ldots, M, 
\end{equation}
and \eqref{harakka} follows from Lemma \ref{separatedisks}. On the other hand, if \eqref{hungary} occurs then Lemma \ref{separatedisks} shows that 
$$
r_i m_1(\Phi_i) \geq \min_{j \leq i+k} \operatorname{dist}(D_i,D_j) - \sum_{j \geq i+k+1} 2r_j  \geq \frac{\alpha r_{i+k}}{10}. 
$$
If moreover $k$ is the smallest index for which \eqref{hungary} occurs, then \eqref{mud} holds for $j$ replaced with $i+k$ and we conclude \eqref{harakka} also in this case. 
Finally, \eqref{harakka2} follows from \eqref{harakka} by a standard application of H\"older's inequality and polar coordinates. 
\end{proof}

Now fix continua $E_1$, $E_2$ as in Proposition \ref{keyestimate}. First, an elementary geometric argument (cf. \cite[Theorem 11.7]{V71}) applying \eqref{fars} shows that there exist $z_0 \in \mathbb{R}^2$ and $r_0>0$ such that both $E_1$ and $E_2$ intersect $S(z_0,t_0r_0)$ for all $1\leq t_0\leq \sqrt{3}$. 

Let $\Phi_i, \Gamma_i$, $i=1,\ldots, M$, be as in Lemma \ref{hanska}. Moreover, we denote $\Phi_0=(1,\sqrt{3})$ and  
\begin{equation}
\label{kuruu}
\Gamma_0=\{S(z_0,t_0r_0): \, t_0 \in \Phi_0\}, 
\end{equation}
so that \eqref{harakka2} holds for all $i=0,\ldots, M$. By construction, $\Gamma_i$ is a family of paths in $\Omega$ when $i=1,\ldots,M$, while the paths in $\Gamma_0$ are not required to lie in $\Omega$. The proposition is proved by modifying $\Gamma_0$ to obtain a family of paths in $\Omega$ such that the lower bound for modulus is still valid. This is based on the following property. 

\begin{lemma}
\label{conne}
Given 
$$
T=(t_0,t_1,\ldots, t_M)\in \Phi := \Phi_0 \times \Phi_1 \times \cdots \times \Phi_M, 
$$ 
there is an injective path $\gamma_T$ connecting $E_1$ and $E_2$ in 
$$
G_T:=\bigcup_{j=0}^M S(z_j,t_jr_j) \cap \Omega. 
$$
\end{lemma}

\begin{proof}
By construction, there are $p_1, p_2 \in S(z_0,t_0r_0)$ such that $p_1 \in E_1$ and $p_2 \in E_2$. On the other hand, 
the components of $S(z_0,t_0r_0) \setminus \Omega$ are of the form 
$$
S(z_0,t_0r_0) \cap D_j = S(z_0,t_0r_0) \cap \overline{D}(z_j,r_j). 
$$
Therefore, the $p_1$-component $F_T$ of $G_T$ contains all of $S(z_0,t_0r_0) \cap \Omega$. In particular, it contains $p_2$. We can choose $\gamma_T$ 
to be a shortest path joining $p_1$ and $p_2$ in $F_T$. 
\end{proof}

Let 
$$
\Gamma=\{\gamma_T: \,  T \in \Phi\}, 
$$ 
where $\gamma_T$ is any path satisfying the conditions of Lemma \ref{conne}. The proposition follows if we can bound $\modu(\Gamma)$ from below. 

Let $\rho \geq 0$ be a Borel function in $\Omega$ such that 
\begin{equation}
\label{dies}
\int_{\Omega} \rho^2 \, dA =\frac{\beta}{(M+2)^2},  
\end{equation}
where $\beta$ is the constant in \eqref{harakka2}. The desired lower bound follows if we can show that such a $\rho$ cannot be admissible for $\Gamma$. By \eqref{harakka2}, \eqref{kuruu}, and \eqref{dies}, we find that $(M+3/2)\rho$ cannot be admissible for any of the path families $\Gamma_i$, $i=0,\ldots,M$. Hence there is at least one $T=(t_0,t_1,\ldots, t_M) \in \Phi$ such that  
$$
\int_{S(z_i,t_ir_i)} \rho \, ds < \frac{1}{M+3/2}  
$$ 
for each $i=0,\ldots, M$. Applying the injectivity of $\gamma_T$, we moreover get 
$$
\int_{\gamma_T} \rho \, ds \leq \sum_{i=0}^M \int_{S(z_i,t_ir_i)} \rho \, ds \leq \frac{M+1}{M+3/2}<1.   
$$
We conclude that $\rho$ is not admissible for $\Gamma$. The proof of Proposition \ref{keyestimate} is complete.


\section{Proof of Theorem \ref{mainthm2}} \label{ilee} 
Suppose that the assumptions of Theorem \ref{mainthm2} are valid. By Proposition \ref{quasiconf}, there is a $\pi/2$-quasiconformal map $f:(X,q) \to \Omega$, where $\Omega \subset \mathbb{S}^2$ is a circle domain. We prove Theorem \ref{mainthm2} by showing that, after a normalization, $f$ is quasisymmetric with respect to the original metric $d$. 

If $\Omega=\mathbb{S}^2$, then the theorem 
is proved in \cite{LRR18}. The proof in the case of one complementary component is easier than the one below and is omitted. 
We consider the remaining case where there are at least two complementary components. 

Recall that the assumptions of V\"ais\"al\"a's theorem (\cite[Theorem 10.17]{H01}) hold in our setting, so the quasisymmetry of $f$ follows if we can prove 
the \emph{weak quasisymmetry} of $h=f^{-1}$: there is $t \geq 1$ such that for every disjoint $y_0,y_1,y_2 \in \Omega$ with 
$$
|y_0-y_1|\leq |y_0-y_2| \leq \frac{1}{10}, 
$$ 
we have 
\begin{equation}
\label{weakqs}
d(h(y_0),h(y_1)) \leq td(h(y_0),h(y_2)). 
\end{equation}

To prove \eqref{weakqs}, we first normalize $h$. Namely, we precompose $h$ with a suitable M\"obius transformation, if necessary, so that 
\begin{equation}
\label{gone}
\min_{i \neq j}d(h(a_i),h(a_j)) \geq \operatorname{diam}_d(X)/10, 
\end{equation} 
where $\{a_0,a_1,a_\infty\} \in \Omega$ correspond to the points $0$, $e_1$, and $\infty$ under the stereographic projection.    

Fix $y_0,y_1,y_2 \in \Omega$ as in \eqref{weakqs}, and denote 
$$
A=d(h(y_0),h(y_2)), \quad B=d(h(y_0),h(y_1)). 
$$ 
We need to show that $B \leq t A$. We may assume that 
$$
A \leq B/100\lambda^3 \leq \operatorname{diam}_d(X)/100\lambda^3, 
$$ 
otherwise there is nothing to prove. By \eqref{gone} and triangle inequality, we find that for some $j \in \{0,1,\infty\}$,  
$$
d(h(y_0),h(a_j))\geq \diam_d(X)/20 \quad \text{and} \quad |y_1-a_j| \geq 1/10. 
$$

Moreover, by the LLC-condition, we find a continuum 
$$
F_1 \subset B_d(h(y_0),\lambda A) \subset X
$$ 
joining $h(y_0)$ and $h(y_2)$. Similarly, we find a continuum 
$$
F_2 \subset X \setminus B_d(h(y_0),B/\lambda)
$$ 
joining $h(y_1)$ and $h(a_j)$. 

Denote $E_\ell=f(F_\ell)=h^{-1}(F_\ell)$ for $\ell=1,2$. Then, if $\tau$ is a rotation of $\mathbb{S}^2$ sending $y_0$ to $0$ and $\phi$ the stereographic projection, we see that $(\phi \circ \tau)(E_1)$ and 
$(\phi \circ \tau)(E_2)$ satisfy the conditions of Proposition \ref{keyestimate}. Since $\phi \circ \tau$ is conformal, we conclude that the lower bound in Proposition \ref{keyestimate} holds also for the continua $E_1$ and $E_2$. 

Next, we estimate $\modu(F_1,F_2;X)$ from above (recall the definition from Section \ref{kuis}). Denote 
$$
U_1=\overline{B}_d(h(y_0),\lambda A), \quad U_2=X \setminus B_d(h(y_0),B/\lambda). 
$$ 
Then, since $F_1 \subset U_1$ and $F_2 \subset U_2$, we have  
$$
\modu(F_1,F_2;X) \leq \modu(U_1,U_2;X). 
$$
Let $k\geq 2$ be the largest integer such that 
$$
B\geq 2^k \lambda^2A,  
$$ 
and denote 
$$
A_j= \overline{B}_d(h(y_0),2^j\lambda A)\setminus B_d(h(y_0),2^{j-1}\lambda A), \quad j=1,\ldots, k. 
$$
The WMDM-condition and the doubling property of $\mu$ then guarantee that for every $\gamma \in \Gamma(U_1,U_2;X)$ the length in $(X,q)$ of the restriction of $\gamma$ to $A_j$ is at least 
$$
\frac{\mu(B_d(h(y_0),2^j\lambda A))^{1/2}}{C_W C_D}. 
$$
It follows that $\rho:U_2 \setminus U_1 \to [0,\infty]$, 
$$
\rho(w)=\frac{1}{k}\sum_{j=1}^k \frac{C_WC_D\chi_{A_j}(w)}{\mu(B_d(h(y_0),2^j\lambda A))^{1/2}} 
$$
is admissible for $\Gamma(U_1,U_2;X)$. Integrating and applying \eqref{nort}, this yields 
\begin{equation} \label{inaali}
\begin{split} 
\modu(U_1,U_2;X) &\leq \frac{C_W^2C_D^2}{k^{2}} \sum_{j=1}^k \frac{\Ha^2_q(A_j)}{\mu(B_d(h(y_0),2^j\lambda A))} \\ 
 &\leq \frac{2\pi C_W^2C_S^2C_D^6}{k}. 
\end{split}
\end{equation}
Finally, combining Proposition \ref{quasiconf}, Proposition \ref{keyestimate}, and \eqref{inaali}, we get 
$$
k \leq \frac{2\pi^3 C_W^2C_S^2C_D^6(10M)^M (M+2)^2}{\alpha^M}, 
$$
where $\alpha$ is the constant in Proposition \ref{keyestimate}. In particular, we have the desired bound for the ratio $B/A$. The proof is complete. 

\vskip 10pt
\noindent
{\bf Acknowledgement.}
We thank the referee for several helpful comments.

\bibliographystyle{abbrv}
\bibliography{viittaukset} 

\vspace{1em}
\noindent
Department of Mathematics and Statistics, University of Jyv\"askyl\"a, P.O.
Box 35 (MaD), FI-40014, University of Jyv\"askyl\"a, Finland.\\

\emph{E-mail:} \settowidth{\hangindent}{\emph{aaaaaaaaa}} Kai Rajala: \textbf{kai.i.rajala@jyu.fi} \\ Martti Rasimus: \textbf{martti.rasimus@gmail.com}

\end{document}